\documentclass[12pt,fleqn]{article}
\usepackage{authblk} 

\usepackage{pslatex}
\usepackage{times}
\usepackage{pifont,amsmath,amsfonts,amssymb,amsthm,amsxtra}
\usepackage{url,mathrsfs,euscript}
\usepackage{graphicx,latexsym,tabularx,enumerate}
\usepackage{dsfont}
\usepackage{epsfig}
\usepackage{hyperref}
\usepackage{pstricks}
\usepackage[overload]{textcase}

\usepackage[all]{xy}
\usepackage{psfrag}
\usepackage{remreset}

\theoremstyle{plain}
\newtheorem{theorem}{Theorem}[section]

\newtheorem{lemma}[theorem]{Lemma}
\newtheorem{proposition}[theorem]{Proposition}

\theoremstyle{definition} 
\newtheorem{definition}[theorem]{Definition}
\newtheorem{remark}[theorem]{Remark}
\newtheorem{remarks}[theorem]{Remarks}
\newtheorem{example}[theorem]{Example}

\def\square{\Box}

\begin{document}
\title{Equivalence of differential equations of order one}

\author[a]{L. X. Chau Ngo}
\author[b]{K. A. Nguyen}
\author[c]{M. van der Put}
\author[d]{J. Top}

\affil[a]{\footnotesize{Quy Nhon University, Department of Mathematics\\

170 An Duong Vuong Street, Quy Nhon City, Binh Dinh Province,Vietnam\\

Email: \url{ngolamxuanchau@qnu.edu.vn}\\

$\ $}}

\affil[b]{\footnotesize{HCMC University of Technology (HUTECH), Department of Computer Science\\

144/24 Dien Bien Phu Str., Ward 25, Binh Thanh Dist., Ho Chi Minh City, Vietnam\\

Email: \url{na.khuong@hutech.edu.vn}\\

$\ $}}

\affil[c,d]{\footnotesize{University of Groningen, Department of Mathematics\\

 P.O. Box 407, 9700 AK, Groningen, The Netherlands\\

Email: \url{mvdput@math.rug.nl}, \url{j.top@rug.nl}}}

\date{}
\maketitle
\begin{abstract}
The notions of equivalence and strict equivalence for order one differential equations of the form
$f(y',y,z)=0$ are introduced. The more explicit notion of strict equivalence is applied to examples and questions
concerning  autonomous equations and equations having the Painlev\'e property. The order one equation $f$
determines an algebraic curve $X$ over $\mathbb{C}(z)$. If $X$ has genus 0 or 1, then it is difficult to verify strict equivalence. However, for higher genus strict equivalence can be tested by an algorithm sketched in the text.  
For autonomous equations, testing strict equivalence and the existence of algebraic solutions are
shown to be algorithmic.
\end{abstract}



\section{Equivalence}
For an irreducible   polynomial $f:=f(S,T,z)\in \mathbb{C}(z)[S,T]$, 
we consider the order one differential equation $f(y',y,z)=0$, where $y'=\frac{dy}{dz}$. The special case that $S$ is not present in $f$ is not really a differential equation and the solutions are algebraic over $\mathbb{C}(z)$. The other special case, namely $T$ is not present in $f$, is still a differential equation. The solutions are the integrals of
finitely many functions which are algebraic over $\mathbb{C}(z)$ (compare \cite{Bro}, \cite{Tra}).  We will exclude these special cases.

In the sequel we will also consider finite extensions $K$ of $\mathbb{C}(z)$, equipped with the unique extension of $\frac{d}{dz}$ to $K$ (which we
also denote as $\frac{d}{dz}$).   Moreover, we will suppose that $f\in K[S,T]$ is absolutely irreducible.\\

For some order one differential equations, like the Riccati equation $y'=ay^2+by+c$, it is easy to describe the solutions. For general $f$ it is difficult to find any information on the solutions and equivalence of equations is a basic theme.

An intuitive way of describing that two such order one differential equation $f_1$ and $f_2$ are equivalent is the existence of an algebraic procedure to obtain from a solution of $f_1$ a solution (or finitely many solutions) of $f_2$ and vice versa.

 In order to make this more precise we have to define what a solution $y$ of $f\in K[S,T]$ is.  First we observe that a solution $y$ of $f(y',y,z)=0$ 
which  is also a solution of $\frac{\partial f }{\partial S }(y',y,z)=0$,
is algebraic over $K$ since the ideal $(f,\frac{\partial f}{\partial S})\subset 
K[S,T]$ has finite codimension as a $K$-vector space. These solutions are easily  computed. Therefore we restrict to solutions $y$ of $f$ such that    $\frac{\partial f }{\partial S }(y',y,z)\neq 0$.\\

We consider the algebra $K[s,t]:=K[S,T]/(f)$ and try to make this into a differential algebra by the derivation $z'=1,t'=s$. Then the derivative $s'$ of $s$ is obtained by differentiation of $0=f(s,t,z)$. Thus
\[0=s'\cdot\frac{\partial f}{\partial s} +s\cdot \frac{\partial f}{\partial t} +\frac{\partial f}{\partial z},\] and we will  restrict to the case that 
\[ d:=\frac{\partial f}{\partial s}
\] is invertible. Then $K[s,t,\frac{1}{d}]$ is a differential algebra.  We note that $\frac{\partial f}{\partial S}$ is called the `separant'  and that the above differential algebra coincides with the `generic solution' of $f$ in the terminology of  \cite{Ri}, p. 129-131.   
In \cite{K}, \S 16 of Chapter IV, related material is considered.

A solution of $f$ is supposed to be algebraic over the field $Mer(U)$ of the meromorphic functions on the universal covering $U$ of an open connected subset of the Riemann surface of $K$. Let $Mer(U)^{a}$ denote the algebraic closure
of the field $Mer(U)$. The differentiation $\frac{d}{dz}$ on $K$  extends uniquely to a differentiation on $Mer(U)^a$.
A {\it solution of $f$} is  a $K$-linear
 homomorphism $$\phi :K[s,t,\frac{1}{d}]\rightarrow Mer(U)^a$$ commuting with differentiation (and $y:=\phi (t)$ is the actual solution).\\

 A variant of the above is the notion of {\it local solution}. This is a $K$-linear differential  homomorphism $$\phi :K[s,t,\frac{1}{d}]\rightarrow \mathbb{C}(\{v^{1/m}\}).$$ The latter is the field of the convergent Laurent series in the variable $v^{1/m}$, where $v$ is a local variable of a
point of the Riemann surface of $K$ and $m\in{\mathbb Z}_{\geq 1}$. Of course a local solution $\phi$ extends to  a solution in $Mer(U)^a$, where $U$ is a small disk around 
a point of the Riemann surface of $K$. On the other hand a solution in  some $Mer(U)^a$ induces local solutions at the points of $U$.  In the following, the precise definition of
solution does not play a role. However, in contrast to \cite{Ri} and \cite{K}, we have chosen for a concrete definition of solution.\\

 In the case that $K[s,t,\frac{1}{d}]$ has only trivial differential ideals, its field of fractions has $\mathbb{C}$ as field of constants (see \cite{vdP-S}) and $\ker \phi =0$. If $\ker \phi \neq 0$, then   $\ker \phi$ is a maximal ideal of $K[s,t,\frac{1}{d}]$. The solution $y=\phi(t)$ is algebraic over $\mathbb{C}(z)$ and is considered as an element of $Mer(U)^a$. There are very few equations admitting an algebraic solution. \\

{\it  It seems to be an open problem whether there exists an algorithm testing the existence of (and computing) algebraic solutions for 
a first order differential equation}.\\
Let $f$ be an order 1 equation. For positive integers $d,n$
one considers algebraic elements $y$ satisfying an equation 
$a_dy^d+a_{d-1}y^{d-1}+\cdots +a_0=0$, where the $a_d,\dots ,a_0$
are polynomials of degree $\leq n$. The coefficients of the $a_j$ are 
seen as variables. Differentiation of this identity yields an expression for
$y'$. The substitution $f(y',y,z)=0$ produces a set of polynomial equations in many variables (over a computable subfield of $\mathbb{C}$).  Gr\"obner  theory provides an algorithm for solving this.
Thus the problem of finding algebraic solutions is `recursive enumerable'.
Missing for a true algorithm are a priori estimates for $d,n$.\\ 

{\it For special cases, there are estimates for $d,n$}. Here are some examples.\\ The solutions
of the  autonomous equation $y'=R(y)$, with $R(y)\in \mathbb{C}(y)$, 
satisfy $\int \frac{dy}{R(y)}=z+c$. The Risch algorithm finds 
the algebraic solutions (if any). An equation like $(y')^2=y^5+1$ yields
Abelian integrals which are transcendental.

 For a Riccati equation $y'+ay^2+by+c=0$ with $a,b,c\in \mathbb{C}(z)$,  Kovacic's algorithm tests the existence and computes possible algebraic solutions. This is done by computing local solutions at the singular points and the observation that a solution $y$ can only have 
poles of order one and integer residue at the non singular points
of the equation.  The above equation has PP, the Painlev\'e property (see \S~\ref{twee} below). 
For every equation with PP there is an algorithm for finding
 algebraic solutions.\\ 

 In contrast to the above, we do not know whether a simple equation like
 $y'=y^3+z$ over $\mathbb{C}(z)$ has algebraic solutions.   The local solutions are:\\
For $z=a \neq \infty $ there is a holomorphic solution $y\in \mathbb{C}\{z-a\}$, depending on the initial value $y(a)$. Moreover there is a  ramified meromorphic solution $y=\sum _{n\geq -1} a_n(z-a)^{n/2}$ in $\mathbb{C}(\{(z-a)\})$, depending on $a_{-1}$ and $a_{-1}^2=\frac{-1}{2}$. \\ 
For $z=\infty$ and with $t:=\frac{1}{z}$ the equation reads 
$\frac{dy}{dt}=-t^2y^3-t^3$. The solutions are $y=\sum _{n\geq -1}
c_nt^{n/3}$ in $\mathbb{C}(\{t^{1/3}\})$, depending on $c_{-1}$ and
$c_{-1}^3=-1$. 

An algebraic solution $y$ has to be ramified at $z=\infty$ of order 3
(and thus $y$ is not rational) and is ramified at some more points with ramification of order 2. However we have no idea what the other ramification points for $y$ could be and what the degree of $y$ over
$\mathbb{C}(z)$ could be. \\

\bigskip

 A criterion for the existence of {\it generic algebraic solutions} is proposed in \cite{A-C-F-G}. For autonomous equations the above criterion leads to an algorithm.  For a first order differential equation $f$, a generic algebraic  solution is a 1-parameter family $\{y_c\}$ of algebraic solutions such that $f$ is the minimal equation for this family. For example, the equation $y'=y^5$ has the generic solution  $y_c^4=\frac{-1}{4z+c}$. First order equations with a generic algebraic solution are very rare.

 \begin{definition}[\textit{Equivalent equations}]\label{Equi_eqns} An {\it equivalence} between equations $f_1$ and $f_2$ is  given by a $\mathbb{C}(z)$-linear differential isomorphism $\psi : F_1\rightarrow F_2$, where for  $j=1,2$, the differential field $F_j$ is a finite extension of the field of fractions of  $\mathbb{C}(z)[S,T,\frac{1}{d_j}]/(f_j)$.
\end{definition}

It is easily seen that the above definition induces an equivalence relation. Let $f_1$ and $f_2$ be
equivalent. Fix $\psi$.  Let $y$ be a solution for $f_1$, given by $$\phi :\mathbb{C}(z)[S,T,\frac{1}{d_1}]/(f_1)\rightarrow Mer(U)^a.$$
Then $\phi$ extends to the field of fractions of $\mathbb{C}(z)[S,T,\frac{1}{d_1}]/(f_1)$ and has finitely many extensions $\phi _1,\dots ,\phi _r$ to $F_1$. The restriction of
$$F_2\stackrel{\psi}{\rightarrow} F_1\stackrel{\phi_j}{\rightarrow} Mer(U)^a$$ to
 $\mathbb{C}(z)[S,T,\frac{1}{d_2}]/(f_2)$ is a solution of $f_2$.

 We conclude that the above definition of equivalence is a way to make the intuitive notion explicit.  It seems rather difficult to decide for explicit $f_1$ and $f_2$ whether they are equivalent. Therefore we  introduce the following notion.

 \begin{definition}[\textit{Strictly equivalent equations}]\label{Strict_qui_eqns} The equations $f_1$ and $f_2$ are {\it strictly equivalent} if there is a finite extension $K$ of $\mathbb{C}(z)$ and a $K$-linear differential isomorphism $\psi$ between the fields of fractions of $K[S,T,\frac{1}{d_1}]/(f_1)$ and $K[S,T,\frac{1}{d_2}]/(f_2)$.
\end{definition}

\begin{remarks}
(i). The $\psi$ in Definition~\ref{Strict_qui_eqns} need not be unique. Indeed, two distinct $\psi$'s differ by a $K$-linear differential automorphism of the field of fractions of $K[S,T,\frac{1}{d_1}]/(f_1)$. The group of the differential automorphism induce a group of permutations of the solutions of $f_1$.\\
(ii). We note that Definitions~\ref{Equi_eqns} and \ref{Strict_qui_eqns} extend in an obvious way to equations $f\in K[S,T]$, where $K$ is
any differential field with field of constants $\mathbb{C}$.\\
(iii). In the sequel we will study strict equivalence for order one differential equations using well known properties of algebraic curves.
\end{remarks}
\section{The Painlev\'e property}\label{twee}

An ordinary differential equation on the complex plane is said to have the \textit{Painlev\'e property} (PP for short) if there are no other moving singularities than poles.

The Painlev\'e property for order one equations has been analysed in detail in \cite{G-vdP}. The reasoning and the results are as follows.

\noindent (1). Observation: {\it If the order one equation $f$ has PP, then $f$ has only finitely many  branched solutions}.\\
 A  branched solution is a solution $\phi : \mathbb{C}(z)[s,t,\frac{1}{d}]\rightarrow \mathbb{C}(\{(z-a)^{1/m}\})$ with $m>1$ and such that $y=\phi (t)$ is not contained in $\mathbb{C}(\{z-a\})$.\\
(2). The field of fractions $F$ of $\mathbb{C}(z)[s,t,\frac{1}{d}]$ is the function field of a smooth projective curve $X$ over $\mathbb{C}(z)$. Let $D$ denote the differentiation on $F$.\\ {\it If the equation $f$ has only finitely many branch points, then every local ring $O_{X,Q}$ (where $Q$ is a closed point) is invariant under $D$}.\\
(3). {\it Suppose that every local ring $O_{X,Q}$ is invariant under $D$. Then there exists a finite extension
$K\supset \mathbb{C}(z)$ and a smooth, connected curve $X_0$ over $\mathbb{C}$, such that
$K\times _{\mathbb{C}(z)}X\cong K\times _{\mathbb{C}}X_0$}. Moreover:\\
 (i) If $X_0\cong \mathbb{P}^1_{\mathbb{C}}$, then $f$ is strictly equivalent to $y'= a_0+a_1y+a_2y^2$ with
 $a_0,a_1,a_2\in K$, not all zero.\\
 (ii) If $X_0$ has genus 1 and equation $y^2=x^3+ax+b$, then $f$ is strictly equivalent to $(y')^2=h\cdot (y^3+ay+b)$
 for some $h\in K^*$.\\
 (iii) If $X_0$ has genus $\geq 2$, then $f$ is strictly equivalent to $y'=0$.\\
(4). Finally, the cases (i)--(iii) have PP.\\

From (1)--(4) one deduces the following.
\begin{proposition}
Let $f_1$ and $f_2$ be strictly equivalent. Then $f_1$ has PP if and only if $f_2$ has PP.
\end{proposition}

An order one equation $f$ is called {\it autonomous} if $f$ is an irreducible element of $\mathbb{C}[S,T]$. A rather difficult question is \textit{whether a given $f$ is strictly equivalent to an autonomous equation.} Let $X$ denote the smooth connected curve over $\mathbb{C}(z)$ such that its function field is the field of fractions of $\mathbb{C}(z)[s,t,\frac{1}{d}]$. We will call $f$ {\it semi-autonomous} if
$K\times _{\mathbb{C}(z)}X\cong K\times _\mathbb{C}X_0$ for some curve $X_0$ over $\mathbb{C}$ and some finite extension $K$ of $\mathbb{C}(z)$. The curve $X$ over $\mathbb{C}(z)$ can be interpreted
as a surface with a  projection  to $\mathbb{P}^1_\mathbb{C}$. In other
words, $X$ has the interpretation of a family of curves over $\mathbb{C}$.
The condition `semi-autonomous' is identical to `$X$ is an isotrivial
family of curves'. 

In the next sections we intend to treat the following items:\\
(i)  The existence of an algorithm deciding whether two curves 
$X_1, X_2$ over a finite extension $K$ of $\mathbb{C}(z)$ become 
isomorphic after a finite extension $L$ of $K$. This includes deciding 
whether a given first order equation is semi-autonomous.\\
(ii)   The existence of an algorithm deciding strict equivalence between two
first order differential equations in case the genus is $\geq 2$.\\
(iii)  The question whether  strict equivalence is
undecidable for the cases of genus 0 and 1.
 
\section{Autonomous equations}

We associate to an irreducible autonomous equation $f(y',y)=0$ (we assume that both $y$ and $y'$ are present in $f$) the pair
$(X,D)$ where the complete, irreducible, smooth  curve $X$ 
has function field $\mathbb{C}(y_1,y_0)$ with equation $f(y_1,y_0)=0$
and $D$ is the meromorphic vector field on $X$ determined by 
$D(y_0)=y_1$. 

\begin{lemma}\label{3.1} Every pair $(X,D)$, consisting of a curve $X/\mathbb{C}$ (complete, irreducible, smooth) and a non zero meromorphic vector field $D$ on $X$,  is associated to some autonomous equation $f(y',y)=0$. 
\end{lemma}
\begin{proof} Let $g\in \mathbb{C}(X)$ satisfy $D(g)\neq 0$. Choose
a closed point $x\in X$ such that $g$ has no pole at $x$, 
$ord_x(g-g(x))=1$ and $ord_x(D(g))=0$. Let $p$ denote a local parameter at $x$. Then $\widehat{O}_{X,x}=\mathbb{C}[[p]]$ and
$ord_x(D(p))=0$. Let $\ell$ be a prime number such that 
$\ell >2\cdot\mbox{genus}(X)+2$ and let $f\in \mathbb{C}(X)$ have a pole of order $\ell$
at $x$ and no further poles. Then $[\mathbb{C}(X):\mathbb{C}(f)]=\ell$.
If $D(f)\not \in \mathbb{C}(f)$, then $\mathbb{C}(f,D(f))=\mathbb{C}(X)$
since $\ell$ is a prime number. 

Suppose that $D(f)\in \mathbb{C}(f)$.
 Then $\frac{D(f)}{f}\in \mathbb{C}(f)\subset \mathbb{C}((\frac{1}{f}))=\mathbb{C}((p^{\ell}))
\subset \mathbb{C}((p))$. This contradicts the fact that 
$ord_x(\frac{D(f)}{f})=-1$. \end{proof}

\begin{remark} {\rm An irreducible order one equation $f(y',y,z)=0$
over a finite field extension $K$ of $\mathbb{C}(z)$ induces a
pair $(X,D)$ of a curve $X$ over $K$ and a derivation $D$ of the function
field of $X/K$. The proof of Lemma 3.1 extends to this non autonomous
case. The statement is: For a given pair $(X,D)$ over $K$, there exists a finite extension $\tilde{K}$ of $K$ and an irreducible order one equation
$f(y',y,z)=0$ over $\tilde{K}$ which induces  
$\tilde{K}\times _KX$ equipped with the unique extension of $D$. } \hfill $\square$ \end{remark}

By $(X,D)$ we denote a pair as in Lemma~\ref{3.1}. Further for any finite extension
$K$ of $\mathbb{C}(z)$ we denote by $K\times (X,D)$ the curve
$K\times _\mathbb{C}X$ with function field $K\otimes _\mathbb{C}
\mathbb{C}(X)$ equipped with the derivation $D^+$ defined by
$D^+=\frac{d}{dz}$  on $K$  and $D^+=D$ on $\mathbb{C}(X)$.
We note that $D^+$ is not a meromorphic vector field since it
is not zero on $K$.

\begin{lemma}\label{3.2} Let
$\phi: K\times (X_1,D_1)\rightarrow K\times (X_2,D_2)$ be an
isomorphism. Then there exists
an isomorphism $(X_1,D_1)\rightarrow (X_2,D_2)$.
\end{lemma}

\begin{proof} The isomorphism $\phi :K\times X_1\rightarrow K\times X_2$ induces an isomorphism $\phi_1:spec(R)\times X_1\rightarrow
spec(R)\times X_2$ for some finitely generated $\mathbb{C}$-algebra
$R$ with field of fractions $K$. After dividing by a maximal ideal of
$R$ we find an isomorphism $X_1\rightarrow X_2$. In the sequel we identify $X_1$ and $X_2$ with some $X$. It is given that some 
automorphism $\phi$ of $K\times X$ has the property
$D_2^+=\phi D_1^+\phi ^{-1}$. We have to show that there exists
an automorphism $\psi$ of $X$ with $D_2=\psi D_1\psi ^{-1}$.

 If the genus of $X$ is $\geq 2$, then $\overline{K}\times X$ and $X$ have the same finite group of automorphisms and there is nothing to prove.\\

Suppose that $X$ has genus zero. Write  $\mathbb{C}(X)=\mathbb{C}(y)$. On the field
$K(y)$ we consider two derivations: $\frac{d}{dy}$ with $\frac{d}{dy}(y)=1$ and $\frac{d}{dy}$
is zero on $K$; further $\frac{d}{dz}$ defined by $\frac{d}{dz}(z)=1$ and $\frac{d}{dz}(y)=0$.
 Let $D_j(y)=f_j(y)\in \mathbb{C}(y)$, then $D_j^+=\frac{d}{dz}+f_j(y)\frac{d}{dy}$ for $j=1,2$. 
 Suppose 
$D_2^+=\phi ^{-1}D_1^+\phi$ where $\phi (y)=\frac{Ay+B}{Cy+D}$ with
${A\ B\choose C\ D}\in {\rm SL}_2(K)$. One computes the identity
\begin{small}
\[f_2(y)=(A'C-AC')(Dy-B)^2+(A'D-AD'+B'C-BC')(Dy-B)(-Cy+A)+\]
\[(B'D-BD')(-Cy+A)^2+(-Cy +A)^2+
(-Cy+A)^2f_1(\frac{Dy-B}{-Cy+A}).\]
\end{small}
A pole $p$ of $f_1\frac{d}{dy}$ yields a pole $\phi (p)$ of $f_2\frac{d}{dy}$ and
so $\phi (p)\in \mathbb{C}\cup \{\infty \}$. If $f_1$ has at least three poles, then $\phi $ is an automorphism of $\mathbb{P}^1_\mathbb{C}$ and we can take $\psi =\phi$.  

If  $f_1\frac{d}{dy}$ has two poles, then the same holds for $f_2\frac{d}{dy}$. We may suppose
that these poles are $0$ and $\infty$ and that $\phi$ has $0$ and $\infty$ as fixed points.
Then $D=A^{-1},\ B=C=0$ and an explicit calculation shows that again $\phi \in {\rm PSL}_2(\mathbb{C})$. A similar calculation can be made for the case that $f_1\frac{d}{dy}$ has at most one pole.\\

Suppose that the genus of $X$ is one and consider $X$ as an elliptic  curve with function field  
$\mathbb{C}(X)=\mathbb{C}(x,y)$ with relation $y^2=x^3+ax+b$. Then $y\frac{d}{dx}$
is the standard invariant derivation on $\mathbb{C}(x,y)$. Let $D_j=f_jy\frac{d}{dx}$ with
$f_j\in \mathbb{C}(x,y)$. We extend $y\frac{d}{dx}$ to $K(x,y)$ by $y\frac{d}{dx}$ is zero on $K$
and introduce $\frac{d}{dz}$ on $K(x,y)$ by $\frac{d}{dz}(z)=1$ and $\frac{d}{dz}$ is zero on
$\mathbb{C}(x,y)$. Then $D^+_j=\frac{d}{dz}+f_jy\frac{d}{dx}$. We are given 
$D_2^+=\phi D_1^+\phi ^{-1}$ and want to prove that there is an automorphism
$\psi$ of $\mathbb{C}(x,y)$ with $D_2=\psi D_1\psi ^{-1}$. We may suppose that $\phi$ is
a translation over the $K$-valued point $(x_0,y_0)$ of $X$.  

Now $\frac{d}{dz}+f_2 y\frac{d}{dx}=\phi \frac{d}{dz}\phi ^{-1} +
\phi f_1 y\frac{d}{dx}\phi ^{-1}$.  Now $\phi f_1y\frac{d}{dx}\phi^{-1}=\phi (f_1)\cdot y\frac{d}{dx}$ and $\phi \frac{d}{dz}\phi ^{-1}=\frac{d}{dz}-\frac{x_0'}{y_0}y\frac{d}{dx}$ (the last formula we found using
an explicit Maple calculation). Suppose $f_1$ has a pole $p$. Then $f_2$ has a
pole $\phi(p)$, since $\frac{x_0'}{y_0}y\frac{d}{dx}$ has no poles. The points $p$ and $\phi(p)$ belong to $X({\mathbb C})$
and therefore the translation $\phi$ is defined over ${\mathbb C}$.
On the other hand, if $f_1$ has no poles, then the same holds for $f_2$ and $c:=f_2-f_1$ is a constant.
Suppose $c\neq 0$. Then $(x_0')^2=c^2y_0^2=c^2(x_0^3+ax_0+b)$. The non constant solutions of this
Weierstrass equation are transcendental (since they are doubly periodic), contradicting the algebraicity of $x_0$.
\end{proof}

By \ref{3.1} and \ref{3.2}, the set of the strict equivalence classes of autonomous first order equations coincides with the set of the equivalence classes of pairs $(X,D)$. We sketch a proof  of the statement:
 {\it There exists an algorithm deciding whether two pairs $(X_j,D_j),\ j=1,2$ are equivalent}. \\

In the next sections we will show that  there is an algorithm deciding whether two curves of the same genus over a fixed algebraically closed field of characteristic zero are isomorphic.
  
  This reduces the problem to the case $X:=X_1=X_2$ and deciding whether there exists an automorphism $\psi$ of $X$ such that 
$D_2=\psi D_1\psi ^{-1}$. If the genus of $X$ is $\geq 2$, then the
finite group of automorphisms of $X$ is computable and this finishes
this case. 

  Suppose that $X$ has genus $\leq 1$, then the automorphism group
of $X$ is infinite. However, the equality  $D_2=\psi D_1\psi ^{-1}$
implies that $\psi$ sends the divisor of $D_1$ to the divisor of $D_2$.
One easily verifies that there is an algorithm for deciding the latter 
condition. Moreover a meromorphic vector field $D$ is determined, up to
a constant, by its divisor.\\

\begin{lemma}\label{3.4} Let the autonomous equation $f$ induce the pair 
$(X,D)$. Then $f(y',y)=0$ has an algebraic solution if and only if
$\mathbb{C}(X)$ contains an element $t$ with $D(t)=1$.  
\end{lemma}
\begin{proof} Let $t\in \mathbb{C}(X)$ satisfy $D(t)=1$. The differential isomorphism $\phi :\mathbb{C}(t)\rightarrow \mathbb{C}(z)$, which sends $t$ to $z+c$ (any constant $c$) extends to a differential embedding of $\mathbb{C}(X)$ into the algebraic closure of 
$\mathbb{C}(z)$. In particular, this produces an algebraic solution for
the equation $f(y',y)=0$.

 On the other hand, suppose that an algebraic solution $y$ exists.
This induces a differential embedding of $\mathbb{C}(X)$ into the
algebraic closure of $\mathbb{C}(z)$. Thus $z$ is an algebraic solution
of the inhomogeneous differential equation $t'=1$ over the differential
field $\mathbb{C}(X)$. Since the differential Galois group of the 
equation $t'=1$ is a finite algebraic subgroup of the additive group $\mathbb{G}_{a}$ one  has $z\in \mathbb{C}(X)$.  \end{proof}
We note that Lemma~\ref{3.4} is essentially present in  \cite{A-C-F-G}.\\

{\it An algorithm for algebraic solutions of the autonomous equation $f(y',y)=0$.}\\
Let the pair $(X,D)$ be induced by $f$. According to Lemma~\ref{3.4}, it suffices to produce an algorithm for finding a solution of $D(t)=1$
with $t\in \mathbb{C}(X)$. Consider a closed point $x\in X$ with local parameter $p$. Then $\widehat{O}_{X,x}=\mathbb{C}[[p]]$. 

A local solution at $x$ has the form $t=a_kp^k+a_{k+1}p^{k+1}+\cdots
\in \mathbb{C}((p))$ and $D(t)=(ka_kp^{k-1}+\cdots )D(p)=1$.

If $D(p)$ has no pole or zero, then $t=a_0+a_1p+\cdots$ with $a_1\neq 0$.

If $D(p)$ has a zero, $D(p)=b_kp^k+\cdots ,\ b_k\neq 0,\ k\geq 1$,
then $k=1$ is not possible and for $k>1$ one has $t$ has a pole of order
$k-1$.

If $D(p)$ has a pole of order $-k$, then $t$ has a zero of order $k+1$.

It follows that a possible $t$ with $D(t)=1$ lies in $H^0(X,L)$ for a 
known line bundle $L$. Testing $D(t)=1$ for the elements of
$H^0(X,L)$ is done by using the Coates algorithm \cite{Coates}.

\section{Strict equivalence for genus 0}
Suppose that $X$ has genus 0. Then $X$ has a rational point since $\mathbb{C}(z)$ is a $C_1$-field. Therefore
$F=\mathbb{C}(z)(u)$ is the function field of $X$ and $u'=g(u,z)$ for some $g(u,z)\in F$. Thus this equation is strictly equivalent  to $f$. \\

It is easily verified that $f$ has PP if and only if $g(u,z)=a_0(z)+a_1(z)u+a_2(z)u^2$.  Indeed, PP is equivalent to
the derivation $D$, given by $D(z)=1,\ D(u)=g(u,z)$, having no poles. \\

Consider the equation $u'=u\cdot G(u,z)$ and $G(u,v)=\alpha _0 \prod (u-\alpha _j)^{n_j}$ with all $\alpha _*$
in a finite extension of $\mathbb{C}(z)$. This equation is strictly equivalent to an autonomous one, i.e., $v'=h(v)$ with $h(v)\in \mathbb{C}(v)$, if and only if $u=\frac{av+b}{cv+d}$ with $a,b,c,d$ in a finite extension of $\mathbb{C}(z)$, $ad-bc\neq 0$ and
\[\frac{u'}{u}=\frac{a'v+b'+ah(v)}{av+b}-\frac{c'v+d'+ch(v)}{cv+d}=\alpha _0 \cdot \prod (\frac{av+b}{cv+d}-\alpha _j)^{n_j}.\]
From this equality one can make a guess for $av+b$ and/or $cv+d$ in case this term is not $1$ and not a multiple of
$z+\beta$ with $\beta \in \mathbb{C}$. This method may solve in some cases the question whether the equation is strictly
equivalent to an autonomous equation.\\

The problem of deciding strict equivalence between two equations $u_1'=g_1(u_1,z)$ and $u_2'=g_2(u_2,z)$ seems to be, like the problem
of finding algebraic solutions, `recursive enumerable'.  Indeed, one has
to investigate whether for some  algebraic $A,B,C,D$ (with $AD-BC=1$) the transformation $u_1\mapsto u_2:=\frac{Au_1+B}{Cu_1+D}$ maps the first equation to the second. For fixed positive integers $d,n$, the Gr\"obner
algorithm produces an answer for $A,B,C,D$ with degree $\leq d$ over
$\mathbb{C}(z)$ and such that the coefficients of their equations over $\mathbb{C}(z)$ are rational functions of degrees $\leq n$. Missing
for a true algorithm is again an a priori bound on $d,n$. A special case
of the problem is the following:\\
 If the equation $u'=g(u,z)$ is strictly equivalent to the equation $v'=0$ by a transformation $u=\frac{Av+B}{Cv+D}$ with $A,B,C,D$ algebraic over $\mathbb{C}(z)$ and $AD-BC=1$, then $u_c:=\frac{Ac+B}{Cc+D}$ is a generic algebraic solution.  In \cite{A-C-F-G} it is shown that a generic algebraic solution exists if and only if the differential polynomial 
$F=u'-g(u,z)$ has zero remainder with respect to a certain  standard differential polynomial  depending upon a number of positive integers. It is remarked in \cite{A-C-F-G} that there is no a priori bound known for these integers if the equation is not autonomous.  This is in accordance with our opinion that there is no algorithm  for the  question whether
equations like $y'=y^3+z$ have algebraic solutions.

A heuristic  indication that no algorithm for finding algebraic  solutions exists is  the order  two equation $(\frac{zy'}{y})'=0$ of low complexity.
One observes namely that the algebraic solutions $z^a$ with $a\in \mathbb{Q}$ have arbitrary complexity.

\begin{remark}[{\it Properties of an autonomous equation of genus 0}]
An autonomous equation with genus 0 has the form $v'k(v)=1$ with $k(v)\in \mathbb{C}(v)^*$ (or is the trivial equation $v'=0$). Write $k(v)=\sum _j \frac{a_j}{v-b_j}+\frac{d}{dv}(k_0(v))$, where the $a_j\not \in \mathbb{Z}$ and $k_0(v)\in \mathbb{C}(v)$. By integration one finds a ``functional equation'' for the solutions, namely
$\sum _ja_j\log (v-b_j)+k_0(v)=z+c$.

If the logarithmic terms  are not present in this formula and the rational function $k_0(v)$ has degree 1, then the only moving singularities of solutions $v$ are poles and the equation has PP.

If there are no logarithmic terms but $k_0(v)$ has degree $>1$, then there are moving branch points for the solutions.

Suppose that $k_0(v)=0$ and thus $z+c=\sum a_j\log (v-b_j)$. Now $z$ has as function of $v$ logarithmic
singularities. One would expect that $v$ has exponential singularities, like $e^{\frac{1}{z-a}}$. In the general case
one expects branch points singularities and singularities of the type $e^{\frac{1}{\sqrt[m]{z-a}}}$
(a mixing of branching and exponential singularities).
\end{remark}

\begin{remark}[{\it Infinitesimal automorphism}] An {\it infinitesimal automorphism} of an order one equation $f$ is, by  definition, a $\mathbb{C}(z)$-linear
derivation $D$ of the field of fractions $F=\mathbb{C}(z)(s,t)$ of $\mathbb{C}(z)[s,t,\frac{1}{d}]$, commuting with the differentiation on $F$. We note that $D$ is determined by $D(t):=h\in F$ and that $h$ should satisfy the equation
$h'=D(s)$ and $D(s)$ is given by the identity $$0=D(f(s,t))=D(s)\cdot \frac{\partial f}{\partial s}
 +h\cdot \frac{\partial f}{\partial t}.$$

For the genus 0 case and $u'=g(u,z)\in F=\mathbb{C}(z)(u)$, the condition on $h:=D(u)$ is
$h\frac{\partial g}{\partial u}-g\frac{\partial h}{\partial u}=\frac{\partial h}{\partial z}$.
For an autonomous equation of genus 0, i.e., $g\in \mathbb{C}(u)$, there exists a non trivial infinitesimal automorphism. Indeed, $D(u)=\lambda g$ with $\lambda \in \mathbb{C}^*$ obviously satisfies the above equation.
  A general equation $f$ of genus 0 has no infinitesimal automorphisms. Indeed,  a computation shows that the equation $u'=a_1u+a_0$, with general $a_0,a_1\in \mathbb{C}(z)$, has no infinitesimal automorphism $D\neq 0$.
\end{remark}

\section{Strict equivalence for genus 1}

Let, for $i=1,2$, the curves $X_i$ with function fields $\mathbb{C}(z)(s_i,t_i)$ associated to the order one differential equation $f_i$, have genus 1.  After a finite extension $K$ of $\mathbb{C}(z)$, the curves have a point $P_i$. The classical method of using the meromorphic functions on $X_i$ with only a pole at $P_i$ yields
$K(s_i,t_i)=K(x,y)$ with $y^2=x^3+a_ix+b_i$ and $a_i,b_i\in K$.
The $j$-invariant classifies elliptic curves over an algebraically closed field. Hence a necessary condition for strict equivalence of $f_1$ and $f_2$ is equality of the $j$-invariants. 

  Suppose that the $j$-invariants 
coincide. Then after replacing $K$ by a finite extension, we may 
identify the function fields of $X_1$ and $X_2$ with $K(x,y),\ 
y^2=x^3+ax+b$. Let $D_1,D_2$ denote the two $\mathbb{C}$-linear derivations on $K(x,y)$, induced by $f_1$ and $f_2$.  Then $f_1$ and $f_2$ are strictly equivalent if and only if there exists 
 a $\overline{K}$-linear automorphism $A$ of
the field $\overline{K}(x,y)$ such that $D_2=AD_1A^{-1}$.    The group of the automorphisms $Aut(E)$ of the elliptic curve $E$ over $\overline{K}$ corresponding to $\overline{K}(x,y)$, has a normal subgroup 
$E(\overline{K})$ of translations  and 
$Aut(E)/E(\overline{K})$ is a cyclic group of order $2,4$ or $6$. 

 As in the case of first order equations of genus 0, the problem of strict equivalence is `recursive enumerable' due to the large group $E(\overline{K})$.  Missing for a true algorithm is an a priori estimate on the degree of the field extension $\tilde{K}$ of $\mathbb{C}(z)$ and of the `height' of the element in $E(\tilde{K})$ needed for a possible automorphism $A$.

\bigskip

Let the order one equation $f(y',y,z)=0$ have genus 1.  If the $j$-invariant is not in $\mathbb{C}$, then $f$ is not strictly equivalent to a semi-autonomous equation.

In the other case, we may suppose that $f$ corresponds, after a finite extension of $K$, to a differential field $K(x,y)$ with $y^2=x^3+ax+b$ with $a,b\in \mathbb{C}$. The differentiation $'$ on the field $K(x,y)$ can be computed and is determined by  $x'=a_0(x,z)+a_1(x,z)y$ (say with
$a_1(x,z)\neq 0$).  Thus we have replaced the original equation $f(y',y,z)=0$ by
$ (\frac{x'-a_0(x,z)}{a_1(x,z)})^2=x^3+ax+b$.  This differential equation 
is far from unique, since it depends on the choice of the  `origin' $P$
of the elliptic curve.  
It seems not possible to decide whether the given equation
$f$ is strictly equivalent to an autonomous equation. However, the verification of PP does not depend on the
particular choice of the point $P$ (see \cite{G-vdP} for details).\\

The difficulty in making strict equivalence explicit for the case that the curve associated to $f$ has genus 0 or 1,
is due to the large group of automorphisms of $X$. We will see that for hyperelliptic curves the situation is different.

\section{Hyperelliptic curves}

Let the pair $(X,D)$, consisting of a curve $X$ over a finite extension $K$ of $\mathbb{C}(z)$ and of a $\mathbb{C}$-linear derivation $D$ of the function field of $X$ satisfying $D(z)=1$, correspond to the order one equation 
$f(y',y,z)=0$. We suppose that the genus $g$ of $X$ is $\geq 2$. \\ An algorithm, due to J.~Coates, computes
for a curve given by an irreducible plane equation,  an explicit basis
of $H^0(X,L)$, where $L$ is any line bundle. In particular, this algorithm
computes an explicit basis
 of the $g$-dimensional vector space $H^0(X,\Omega _{X/K})$ over $K$ (in fact for a number field $K$; the function field case is similar).
For a  closed point $x$ of $\overline{K}\times_K  X$, one chooses  a local parameter $t$
and considers the map $\overline{K}$-linear map $\ell _x:\overline{K}\otimes _KH^0(X,\Omega _{X/K})\rightarrow \overline{K}$, given by $\ell_x(\omega)=a$ if, locally at $x$, one has $\omega =fdt$ and $f(x)=a$.   
A change of $t$ has the effect of multiplying $\ell_x$ by an element in $\overline{K}^*$. 
This yields an algorithmic description of  the  canonical morphism $\phi :X\rightarrow \mathbb{P}(H^0(X,\Omega _{X/K})^*)\cong \mathbb{P}^{g-1}_K$.\\

For genus two, $\mathbb{P}(H^0(X,\Omega _{X/K})^*)$ is the projective line over $K$ and we obtain an explicit
degree two map $\phi: X\rightarrow \mathbb{P}^1_K$.  This leads to an explicit equation $y^2=P(x)$, with
$P(x)\in K(x)$, where $x$ is a parameter for the projective line over $K$. Since the genus is two, one can take
$P(x)$ to be  a separable polynomial of degree six.\\

If the genus $g$ is $>2$, then for a `general' curve $X$, the morphism $\phi$ is the canonical embedding of
$X$ into  $\mathbb{P}^{g-1}_K$. The curves for which $\phi$ is not an embedding are called {\it hyperelliptic}.
It is known that (see, e.g., \cite{vL-vdG}), in that case, the image of $\phi$ in $\mathbb{P}_K^{g-1}$ is a genus zero curve, called $(g-1)$-uple curve. Since the field $K$ is a $C_1$-field, the genus zero curve is isomorphic to $\mathbb{P}^1_K$.  
The $(g-1)$-uple curve is an embedding $\mathbb{P}^1_K\rightarrow \mathbb{P}_K^{g-1}$, given 
 by $(x_0:x_1)\in \mathbb{P}_K^1\mapsto (x_0^{g-1}:x_0^{g-2}x_1,\dots , :x_1^{g-1})\in \mathbb{P}_K^{g-1}$,
in suitable coordinates.  The resulting morphism $X\rightarrow \mathbb{P}^1_K$ has degree two. The curve $X$ over $K$ is then 
represented by an explicit  equation $y^2=P(x)$ where $P(x)\in K[x]$
can be chosen to be  a separable polynomial of degree $2g+2$.

The main observation is the existence of an algorithm computing
an equation $y^2=P(x)$, with separable $P(x)\in K[x]$ of degree $2g+2$, for a curve $X$ over $K$ of genus $g\geq 2$ which is known to be  hyperelliptic.  Moreover, the divisor of $P(x)$ in $\mathbb{P}^1_K$ is unique up to automorphisms of $\mathbb{P}^1_K$.\\
 
{\it Testing (semi-)autonomous}.\\
Let the pair $(X,D)$ be derived from the order one equation $f(y',y,z)=0$
and suppose that $X$ is a hyperelliptic curve over $K$ of genus $g\geq 2$. Let $y^2=P(x)$ with $P(x)\in K(x)$ a separable polynomial of degree $2g+2$ and let $R$ denotes its divisor.  
Suppose that $\overline{K}\times_{K} X$ is isomorphic to $\overline{K}\times _{\mathbb{C}}X_0$. Then, as above, one obtains an equation $y^2=v$, with $v\in \mathbb{C}[x]$ a separable polynomial of degree $2g+2$, for $X_0$. Its divisor $\tilde{R}$ on $\mathbb{P}^1_{\mathbb{C}}$ is unique up to automorphisms of $\mathbb{P}^1_{\mathbb{C}}$.

  The isomorphism between $\overline{K}\times_{K} X$ and $\overline{K}\times _{\mathbb{C}}X_0$ induces an
  isomorphism between the two projective lines $\mathbb{P}^1_{\overline{K}}$ and $\overline{K}\times \mathbb{P}^1_\mathbb{C}$ which
  sends the divisor $R$ to $\tilde{R}$. We conclude the following.

\begin{proposition}\label{6.1} The equation $f$ is semi-autonomous (i.e., there is an isomorphism
  $\overline{K}\times_{K} X\rightarrow \overline{K}\times _{\mathbb{C}}X_0$) if and only if  there exists an element
  $A\in {\rm PGL}_2(\overline{K})$ such that the divisor $AR$ is defined over the subfield $\mathbb{C}$ of $\overline{K}$.
\end{proposition}
It is easy to verify the existence of $A$ in Proposition~\ref{6.1}. After a finite extension of $K$, we may suppose that 
$P(x)=\prod _{r\in R}(x-r)$ where $R\subset K$ has cardinality $2g+2$. Then one defines $A$ by, say, $A$ maps three 
distinct elements $r_1,r_2,r_3$ of $R$ to  1,2,3. 
Then $f$ is semi-autonomous if and only if
   $A(R)\subset \mathbb{P}^1(\mathbb{C})$.

  Suppose that $f$ is semi-autonomous. Then one computes on the field $K(X)$, which is identified
  with $K(y,x)$ with $y^2=v\in \mathbb{C}[x]$ (as above), the action of the differentiation $D$. Then $f$ is strictly equivalent to an autonomous equation if and only if $D(x)\in \mathbb{C}(x,y)$.\\

{\it Testing strict equivalence}.\\
 An algorithm for testing strict equivalence between two equations $f_1$ and $f_2$ of genus $g\geq 2$ can be obtained in a similar way. After a finite extension $K$ of $\mathbb{C}(z)$, the two fields are given for $j=1,2$ by equations $y^2=\prod _{r\in R_j} (x-r)$ where $R_1,R_2$ are subsets of $\mathbb{P}^1(K)$ of cardinality
$2g+2$. The two fields are isomorphic if and only if some automorphism $A$ of $\mathbb{P}^1_K$ which maps 
three chosen elements of $R_1$ to three chosen elements of $R_2$, has the property $A(R_1)=R_2$. 
 
  If $A$ exists then we identify the two fields with the field corresponding
to $y^2=\prod _{r\in R}(x-r)$, where $R\subset \mathbb{P}^1(K)$ consists of $2g+2$ elements. The automorphism group of this field is very explicit. It is generated by the involution $y\mapsto -y,\ x\mapsto x$ and the finite group of the  automorphism of $\mathbb{P}^1_K$ preserving the set $R$.

 Let $D_1,D_2$ denote the two derivations of this field coming from the equations $f_1,f_2$.   Then $f_1$ and $f_2$ are strict equivalent if and
only if there exists an automorphism $A$ with $D_2=AD_1A^{-1}$.

\begin{example} The ``standard'' autonomous equation for a genus two curve is $$ (y')^2-\prod_{1}^{6}(y-a_j) =0, $$ where the $a_j$ are distinct elements of $\mathbb{C}$. The function field of the equation (over $\mathbb{C}(z)$) is $ F=\mathbb{C}(z)(s,t),$ with equation $s^2-\prod_{1}^{6}(t-a_j) =0,$ and the differentiation is given by $t'=s$ and $s'=\frac{1}{2}\frac{dP}{dt},$ where $P(t)=\prod_{1}^{6}(t-a_j).$

One can put the standard equation in disguise by choosing a $T\in F$ such that $F=\mathbb{C}(z)(S,T),$ with $S=T'.$

\noindent (a). The choice $T=f\cdot t$ with $f\in\mathbb{C}(z)^*$ leads to the equations
$$ f^2(fS-f'T)^2- \prod_{1}^{6}(T-fa_j)=0.$$
For example, if we take $f=z^{-1}$ then the equation above becomes $$ (zS+T)^2- \prod_{1}^{6}(zT-a_j)=0. $$

\noindent (b). The choice $T=f\cdot s$ with $f\in\mathbb{C}(z)^*$ is also possible. First we consider the case $f=1$. Then $\mathbb{C}(s,s')=\mathbb{C}(s,t)$. Indeed, we have $s^2=P,$ $\frac{dP}{dt}\in\mathbb{C}(s,s')$ and $\mathbb{C}(P,\frac{dP}{dt})=\mathbb{C}(t).$

Let $G\in\mathbb{C}[X_1,X_2]$ be the irreducible polynomial satisfying
\begin{equation*}G(s,s')=0. \label{eq:autonomous} \tag{$\star$}
\end{equation*}
Then this is another autonomous equation for our genus two field $\mathbb{C}(s,t).$

Now the choice $T=f.s$ and $T'=S$ produces the order one differential equation

$$ G\left(\frac{S-\frac{f'}{f}T}{f},\frac{T}{f}\right)=0. $$

We can make the equation (\ref{eq:autonomous}) explicit as follows. There is a rational function $Q(X_1,X_2)\in\mathbb{C}(X_1,X_2)$ such that $t=Q(P,\frac{1}{2}\frac{dP}{dt}).$ The the equation $G$ between $s'$ and $s$ is obtained from $s^2=P(t)=P(Q(s^2,s')).$

As an example, consider $s^2=P(t)=t^6-1$. Then we get $t=\frac{6P}{\frac{dP}{dt}},$ and therefore the equation  $s^2=t^6-1$ becomes $$ G(s,s')=s^2(s')^6-6(s^2)^6+(s')^6=0.$$

\end{example}

\section{Genus three and  non hyperelliptic}

{\it Testing strict equivalence}.\\
Suppose that the order one differential equations $f_1,f_2$ define 
the pairs $(X_j,D_j),\ j=1,2$ consisting of a genus 3 curve over $K$
which is not hyperelliptic and a $\mathbb{C}$-linear derivation  on the function field of this curve, satisfying $D_j(z)=1$.   

First one wants to investigate whether
the curves become isomorphic after a finite extension of $K$. The curve
$X_j$ has a canonical embedding as a smooth curve  in $\mathbb{P}^2_K$, given by a homogeneous polynomial $F_j$
 of degree 4 (unique up to constants).  Further $\overline{K}\times_K X_1$ is isomorphic to $\overline{K}\times_K X_2$ if and only if there exists an automorphism $A$ of  $\mathbb{P}^2_{\overline{K}}$ such that $F_2=AF_1$ (up to constants).\\

 Let $Z$ denote the variety of the homogeneous polynomials of degree 4 (up to multiplication by constants and defining a smooth curve). On this variety
the group ${\rm PGL}_3$ acts in a natural way.  The naive quotient $Z/{\rm PGL}_3$
does not exist. However, one can compute generators $\{I_k\}$ for the ring of the ${\rm PGL}_3$-invariant homogeneous functions on $Z$. This defines a `coarse moduli space'.
For a base field  which is algebraically closed (in our case $\overline{K}$), two 
degree 4 homogeneous polynomials $F_1,F_2$ are equivalent 
(up to constants)  under ${\rm PGL}_3$ if and only if the basic invariants $\{I_k\}$ have the same values in $F_1$ and $F_2$. An explicit computation of the $\{I_k\}$ is given in \cite{R}. Moreover, this
thesis contains an algorithm which produces $A$ with $F_2=AF_1$ (up to constants), whenever $ I_k(F_1)=I_k(F_2)$ for all $k$.\\

Another interesting method testing whether $X_1,X_2$ become isomorphic over an extension of $K$ can be deduced from \cite{F}, Proposition 1.1. The idea is that one provides the
smooth degree 4 curves with an additional structure  such that the 
new space $Z^+$, consisting of these curves with extra structure, admits a good quotient by ${\rm PSL}_3$ (which is  a `fine moduli space' for the problem 
considered here).

 This additional structure consists, for a degree 4 smooth
curve $X\subset \mathbb{P}^2$, of two bitangents $L_1,L_2$ and the tangent points
$A_1,B_1$ on $L_1$ and the tangent points $A_2,B_2$ on $L_2$. The map
$Z^+\rightarrow Z$, which forgets the extra structure, is finite surjective and  has degree 
$28\times 27\times 2\times 2$, since a smooth degree 4 curve has 28 bitangents.

Now computing a possible isomorphism $\overline{K}\times _KX_1\rightarrow   
\overline{K}\times _KX_2$ can be done as follows. Let $F_1$ be the equation of the embedded $X_1$ and (after an extension of $K$) we choose 
$(L_1,L_2,A_1,B_1,A_2,B_2)$. Let $F_2$ be the equation for $X_2$ and consider any of the possible tuples $(L^*_1,L^*_2,A^*_1,B^*_1,A^*_2,B^*_2)$  for $F_2$. Let 
$\phi \in {\rm PGL_3}$ be the unique transformation $\phi$ which $(A_1,B_1,A_2,B_2)\mapsto  (A^*_1,B^*_1,A^*_2,B^*_2)$.  One computes whether
 $\phi F_1=F_2$ (up to scalars).  If this has no success for any of the possible tuples,
then   $\overline{K}\times _KX_1$ is not isomorphic to  $\overline{K}\times _KX_2$. 

Suppose now that $X_1$ can be identified with $X_2$ (after replacing $K$ by a finite 
extension).  Then for strict equivalence
one has to test whether $D_2=AD_1A^{-1}$ holds for some element $A$ in the known
finite group of automorphisms of $X_1=X_2$ (again possibly extending $K$).\\

{\it Testing strict equivalence  to a (semi-)autonomous equation}.\\
Let $(X,D)$ denote a curve of genus 3 over $K$ which is not hyperelliptic
and $D$  a derivation of the function field of $X$ such that $D(z)=1$.
As before, $X$ yields a homogeneous polynomial $F$ of degree 4.
The curve is semi-autonomous if and only if the values of the invariants $I_k$ for $F$ are in $\mathbb{C}$. If $X$ is semi-autonomous, then
according to \cite{R}, there is an algorithm producing 
$A\in {\rm PGL}_3(\overline{K})$ such that $F_0:=A(F)$ has its coordinates in $\mathbb{C}$. The homogeneous polynomial  $F_0$
of degree 4  defines a curve $X_0$ over $\mathbb{C}$ such that $\overline{K}\times _KX\cong \overline{K}\times _\mathbb{C}X_0$. Then $(X,D)$ is strictly equivalent to an autonomous equation if and only if $D$ leaves the function field of $X_0$ invariant. 

\section{Non hyperelliptic curves of higher genus}
 Let, as before, the pair $(X,D)$ correspond to an order one differential
equation. Suppose that $X/K$ has genus $g\geq 3$ and that $X$
is not hyperelliptic. Testing strict equivalence and equivalence to a
(semi-)autonomous equation can be done as in \S 6  if one has a reasonable  explicit  (coarse or fine) moduli space for non hyperelliptic curves of genus $g$ and a way of determining the finite group of
automorphisms of a given curve of this type. \\

 The following method shows the existence of an algorithm based upon properties of the  Weierstrass 
points of a curve. Let $X_j,\ j=1,2$ denote non hyperelliptic curves of genus 
$g\geq 3$ over $K$. The canonical embeddings 
$X_j\subset \mathbb{P}^{g-1}_K$ are explicit. Let 
$W_j\subset \overline{K}\times _KX_j$ denote the finite (and effectively computable) set of
Weierstrass points. After a finite extension of $K$ we may suppose that the points of $W_1,W_2$
are $K$-rational. An isomorphism $\phi :\overline{K}\times X_1
\rightarrow \overline{K}\times X_2$ is induced by a unique 
automorphism $\psi$ of $\mathbb{P}^{g-1}_{{K}}$ which maps $W_1$ to $W_2$.  There are only finitely many 
automorphisms $\psi$ such that $\psi (W_1)=W_2$. One can test these $\psi$'s for the properties $\psi (X_1)=X_2$
and  ${\psi^*}^{-1}D_2\psi^*=D_1$. This yields an algorithm as desired.

\section*{Acknowledgements} A preliminary  version of this paper was written while the first three authors were guests of Vietnam Institute for Advanced Study in Mathematics (VIASM) in Ha Noi. They thank VIASM for its hospitality.
We thank Frans Oort for his interest in this work.

\end{document}